\documentclass[a4paper,12pt]{article}

\usepackage{color}
\usepackage[pdftex]{graphicx}
\usepackage{amsmath,amssymb,amsthm}
\usepackage{wrapfig}
\usepackage{ifpdf}

\newtheorem{theorem}{Theorem}
\newtheorem{cor}{Corollary}

\newcommand{\conv}{\mathop{\rm conv}\nolimits}

\title {Kakutani type Borsuk--Ulam theorem}

\author {Oleg R. Musin}

\begin{document}

	\ifpdf \DeclareGraphicsExtensions{.pdf, .jpg, .tif, .mps} \else
	\DeclareGraphicsExtensions{.eps, .jpg, .mps} \fi	
	
\date{}
\maketitle

\begin{abstract}  In this paper we present a Kakutani type theorem that is equivalent to the Borsuk--Ulam theorem for manifolds. 
\end{abstract}

\medskip









For a set $S$ denote by $2^S$ the set of all subsets of $S$, including the empty set and $S$ itself. Let $X$ and $Y$ be topological spaces. Say that $F:X\to 2^Y$ is a {\it set-valued function on $X$ with a closed graph} if graph$(F)=\{(x,y)\in X\times Y: y\in F(x)\}$ is a closed as a subset of $X\times Y$. 

 The Kakutani fixed point theorem \cite{Kakutani} states:\\ {\it Let $S$ be a non-empty, compact and convex subset of ${\Bbb R}^d$. Let $F:S\to 2^S$ be a set-valued function on $S$ with a closed graph and the property that $F(x)$ is non-empty and convex for all $x\in S$. Then there is $y\in S$ such that $y\in F(y)$. (The point $y$ is called a fixed point.)}



\medskip

In our paper \cite{MusBUT} we extended the Borsuk--Ulam theorem for manifolds.\\  
Let  $M$ be a compact PL (piece-wise linear) $d$-dimensional manifold without boundary with a free simplicial involution $A:M\to M$, i. e. $A^2(x)=A(A(x))=x$ and $A(x)\ne x$. 
We say that a pair $(M,A)$ is a {\it BUT (Borsuk-Ulam Type) manifold} if for any continuous  $g:M \to {\Bbb R}^d$ there is a point $x\in M$ such that $g(A(x))=g(x)$. Equivalently, if a continuous  map $f:M \to {\Bbb R}^d$  is { antipodal}, i. e. $f(A(x))=-f(x)$,  then the zeros set $Z_f:=f^{-1}(0)$ is not empty. 

In \cite{MusBUT}, we found several equivalent necessary and sufficient conditions for manifolds to be BUT. For instance, {\it $M$ is a BUT manifold if and only if $M$ admits an antipodal continuous transversal to zeros map $h:M \to {\Bbb R}^d$ with $|Z_h|=2\pmod{4}$.} 

The class of BUT manifolds is sufficiently large. It is clear that $({\Bbb S}^d,A)$ with $A(x)=-x$ is a BUT-manifold.   Suppose that $M$ can be represented as a connected sum $N\# N$, where $N$ is a closed PL manifold. Then $M$ admits a free involution. Indeed, $M$ can be  ``centrally symmetric" embedded to ${\Bbb R}^k$ with some $k$ and the antipodal symmetry $x\to -x$ in ${\Bbb R}^k$ implies a free involution $A:M\to M$ \cite[Corollary 1]{MusBUT}. For instance,   orientable  two-dimensional  manifolds $M^2_g$ with even genus $g$ and non-orientable manifolds $P^2_m$ with even $m$, where $m$ is the number of  M\"obius bands,  are BUT-manifolds.

\begin{theorem}
Let $M$ be a compact PL BUT--manifold  of dimension $d$ with a free involution $A$. 
 Let $F:M\to 2^{{\Bbb R}^d}$ be a set-valued function on $M$ with a closed graph and the property that for all $x\in M$ the set $F(x)$ is non-empty and convex in ${\Bbb R}^d$ and  there is $y\in F(x)$ such that $(-y)\in F(A(x))$. Then there is $x_0\in M$ such that  $F(x_0)$ covers the origin $0\in{\Bbb R}^d$.
\end{theorem}
\begin{proof} Note that any PL manifold admit a metric. For a triangulation $T$ of $M$ the norm of $T$, denoted by $|T|$, is the diameter of the largest simplex in $T$.
	
	Let $\{T_i,\, i=1,2,\ldots\}$ be a sequence of antipodal triangulations of $M$ with $|T_i|\to0$. Now we define an antipodal mapping $f_i:M\to {\Bbb R}^d$.  Note that if we define $f_i$ for all vertices $V(T_i)$, then $f_i$ can be piece-wise linearly extended to all $T_i$, i. e.  to $f_i:M\to {\Bbb R}^d$. 
	
	 Let $x\in V(T_i)$. Then $A(x)$ is also a vertex of $T_i$. Let $y\in F(x)$ be a point in ${\Bbb R}^d$ such that $(-y)\in F(A(x))$. Then for the pair $(x,A(x))$ set 
$f_i(x):=y$ and $f_i(A(x)):=-y$. 

Now we have a continuous antipodal mapping $f_i:M\to {\Bbb R}^d$. By assumption, $M$ is BUT, therefore, there is a point $x_i\in M$ such that $f_i(x_i)=0$. Now, suppose that $x_i$ lies in a $d$-simplex $s_i$ of $T_i$ with vertices $v_0^i,\dots,v_d^i$  and let $y_k^i:=f_i(v_k^i)$. (Note that $y_k^i\in F(v_k^i)\subset{\Bbb R}^d$.) We have $0\in \conv(y_0^i,\dots,y_d^i)$ in ${\Bbb R}^d$. Let $t_0^i,\dots,t_d^i$ 
be the barycentric coordinates of $0$ relative to the simplex  in ${\Bbb R}^d$ with vertices $y_0^i,\dots,y_d^i$. Then  
$$
\sum\limits_{k=0}^d{t_k^iy_k^i}=0.
$$

Since $M$ is compact the sequences $\{x_i\}_{i=1}^\infty,\, \{t_k^i\}_{i=1}^\infty,\, \{y_k^i\}_{i=1}^\infty$, $k=0,1,\ldots,d$, may, after possibly renumbering them, be assumed to converge to points $x_0, t_k$ and $y_k, \, k=0,1,\ldots,d$ respectively. 

Finally, since the set-valued function $F$ has closed graph, $y_k\in F(x_0),\, k=0,1,\ldots,d$. Since
$F$ takes convex values, $F(x_0)$ is convex and so  $0\in{\Bbb R}^d$, being a convex combination of the $y_k$ is in $F(x_0)$. Thus $x_0$ is the required point in $M$.
\end{proof}

\begin{cor} Let $M$ be a  compact PL BUT--manifold of dimension $d$ with a free involution $A$. 
 Let $F:M\to 2^{{\Bbb R}^d}$ be a set-valued function on $M$ with a closed graph.   Then there is $x_0\in M$ such that the sets $F(x_0)$ and $F(A(x_0))$ have a common intersection point.
\end{cor}
\begin{proof} Let $\bar F(x), \, x\in M,$ be the Minkowski difference of $F(x)$ and $F(A(x))$ in ${\Bbb R}^d$, 
$$
\bar F(x):=F(x)-F(A(x))=\{a-b| a\in F(x), b\in F(A(x)).
$$
Then $\bar F(A(x))=-\bar F(x)$. Thus, the theorem for $\bar F$ yields this corollary. 
\end{proof}

\end{document}